\documentclass[10pt,a4paper]{elsarticle}

\journal{Disc. Applied Math.}

\usepackage{hyperref}
\usepackage{amsmath,amsthm,enumerate}
\usepackage{amssymb}
\usepackage{graphicx}
\usepackage{tikz}
\usetikzlibrary{calc,decorations.pathmorphing,decorations.text,fixedpointarithmetic}
\usepackage{subfigure}
\usepackage{refcount}
\usepackage[hmargin=2.5cm,vmargin=3cm]{geometry}
\usepackage{algpseudocode,algorithm,algorithmicx}
\algrenewcommand\algorithmicrequire{\textbf{Input:}}
\algrenewcommand{\algorithmiccomment}[1]{\hfill{\color{gray}\# #1}}
\newcommand*\Let[2]{\State #1 $\gets$ #2}

\newtheorem{theorem}{Theorem}[section]
\newtheorem{proposition}[theorem]{Proposition}
\newtheorem{lemma}[theorem]{Lemma}
\newtheorem{observation}[theorem]{Observation}
\newtheorem{conjecture}[theorem]{Conjecture}

\newtheorem{problem}[theorem]{Problem}

\theoremstyle{definition}

\newcommand{\SBSPG}[1]{\mathcal{SBSPG}_{#1}}
\newcommand{\vphi}{\varphi}

\usepackage[color=green!30]{todonotes}

\begin{document}

\begin{frontmatter}
  \title{Homomorphism bounds of signed bipartite $K_4$-minor-free graphs and edge-colorings of $2k$-regular $K_4$-minor-free multigraphs\tnoteref{t1}}

\tnotetext[t1]{This work is supported by the IFCAM project Applications of graph homomorphisms (MA/IFCAM/18/39) and by the ANR project HOSIGRA (ANR-17-CE40-0022). The third author has also received financial support from UMI3457 of CNRS (France) at CRM (Montreal, Canada) towards the completion of this project.}
 
 \author[LIMOS]{Laurent Beaudou}
 \author[LIMOS]{Florent Foucaud}
 \author[IRIF]{Reza Naserasr}

 \address[LIMOS]{LIMOS - CNRS UMR 6158, Universit\'e Clermont Auvergne, Clermont-Ferrand (France).\\ E-mails: laurent.beaudou@uca.fr, florent.foucaud@gmail.com}
 \address[IRIF]{CNRS - IRIF UMR 8243, Universit\'e Paris Diderot - Paris 7 (France). E-mail: reza@irif.fr}

 \begin{abstract}
A signed graph $(G, \Sigma)$ is a graph $G$ and a subset $\Sigma$ of its edges which corresponds to an assignment of signs to the edges: edges in $\Sigma$ are negative while edges not in $\Sigma$ are positive. A closed walk of a signed graph is balanced if the product of the signs of its edges (repetitions included) is positive, and unbalanced otherwise. The unbalanced-girth of a signed graph is the length of a shortest unbalanced closed walk (if such a walk exists). A homomorphism of $(G,\Sigma)$ to $(H,\Pi)$ is a homomorphism of $G$ to $H$ which preserves the balance of closed walks.

  In this work, given a signed bipartite graph $(B, \Pi)$ of unbalanced-girth $2k$, we give a necessary and sufficient condition for $(B, \Pi)$ to admit a homomorphism from any signed bipartite graph of unbalanced-girth at least $2k$ whose underlying graph is $K_4$-minor-free. The condition can be checked in polynomial time with respect to the order of $B$.

  Let $SPC(2k)$ be the signed bipartite graph on vertex set $\mathbb{Z}_2^{2k-1}$ where vertices $u$ and $v$ are adjacent with a positive edge if their difference is in $\{e_1,e_2, \ldots, e_{2k-1}\}$ (where the $e_i$'s form the standard basis), and adjacent with a negative edge if their difference is $J$ (that is, the all-1 vector). As an application of our work, we prove that every signed bipartite $K_4$-minor-free graph of unbalanced-girth $2k$ admits a homomorphism to $SPC(2k)$. This supports a conjecture of Guenin claiming that every signed bipartite planar graph of unbalanced-girth $2k$ admits a homomorphism to $SPC(2k)$ (this would be an extension of the four-color theorem).

  We also give an application of our work to edge-coloring $2k$-regular $K_4$-minor-free multigraphs.
  \end{abstract}


\end{frontmatter}

\section{Introduction}

In a recent work~\cite{BFN17}, we gave a necessary and sufficient condition for a graph of odd-girth $2k+1$ to admit a homomorphism from any $K_4$-minor-free graph of odd-girth at least $2k+1$. We also presented applications of this result to edge-coloring of $(2k+1)$-regular $K_4$-minor-free multigraphs. Recently, Naserasr, Rollov\'a and Sopena~\cite{NRS15} introduced the notion of homomorphisms of signed graphs, as an extension of classic graph homomorphisms. They showed that the classic graph homomorphism questions are captured by the restriction of homomorphisms to the class of signed \emph{bipartite} graphs, thus providing motivation to study homomorphisms on this subclass.

In this paper, we provide a theorem that, for a given signed bipartite graph $(B,\Pi)$ of unbalanced-girth $2k$, helps to decide whether every signed bipartite $K_4$-minor-free graph admits a homomorphism to $(B,\Pi)$. Indeed, our theorem provides a necessary and sufficient condition on $(B,\Pi)$ for when this is the case. Proving that signed projective cubes of odd dimension satisfy the latter condition, we show that they are examples of such homomorphism bounds. As an application, we consider edge-coloring $2k$-regular $K_4$-minor-free multigraphs. Our work is motivated by some conjectures extending the four-color theorem (these are discussed in~\cite{NRS13}) and it provides partial support to these conjectures.

We start by giving our notation in Section~\ref{sec:notation}. Then, we describe the main property used in our characterization of bounds in Section~\ref{sec:mainprop}. We then state and prove our main characterization theorem in Section~\ref{sec:mainthm}. We present some algorithms based on this theorem in Section~\ref{sec:algo}, and we give further appplications of it in Section~\ref{sec:application}. Finally, we discuss our results and conclude the paper in Section~\ref{sec:conclu}.

\section{Notation}\label{sec:notation}

Dealing with signed graphs requires some technical notations. We gather them in this section. The term {\em graph} refers to simple and finite graphs (no loop or multi-edge). We do not consider loops in this work, but multi-edges are important. When needed, a graph with possible multi-edges will be referred to as a {\em multigraph}.

Given two graphs $G$ and $H$, a mapping $\vphi: V(G) \rightarrow V(H)$ is a {\em homomorphism from $G$ to $H$} if, for any edge $uv$ in $G$, $\vphi(u)\vphi(v)$ is an edge in $H$. Homomorphisms generalize the notion of proper coloring since a $k$-coloring is nothing but a homomorphism to $K_k$.

A graph $H$ is a {\em minor} of a graph $G$ if it can be obtained from $G$ by a sequence of these operations: deleting edges or vertices, and contracting edges. If a graph $H$ is not a minor of a graph $G$, we say that $G$ is {\em $H$-minor-free}. In this work, we mainly focus on $K_4$-minor-free graphs. These graphs are also known as \emph{partial $2$-trees}.

\subsection{Partial $k$-trees}

For any fixed integer $k$, the class of {\em $k$-trees} is defined recursively as follows:
\begin{itemize}
  \item the complete graph $K_k$ is a $k$-tree,
  \item given a $k$-tree $H$ and a $k$-clique $K$ of $H$, the graph obtained from $H$ by adding a vertex adjacent exactly to the vertices of $K$ is a $k$-tree.
\end{itemize}
A {\em partial $k$-tree} is any subgraph of a $k$-tree. Trees are $1$-trees (and reciprocally) and the partial $1$-trees are exactly the forests. The classes of partial $k$-trees are closed under taking minors (the partial $1$-trees are exactly the $K_3$-minor-free graphs and the partial $2$-trees are exactly the $K_4$-minor-free graphs, as shown by Wald and Colbourn~\cite[Theorem 3.1]{WC83}). In general, partial $k$-trees form a subclass of $K_{k+2}$-minor-free graphs. In particular, partial $3$-trees are $K_5$-minor-free, however not every partial $3$-tree is planar.

Partial $k$-trees have been studied under various equivalent viewpoints and are most notably known as the class of graphs of tree-width at most $k$. Partial $2$-trees are also known as {\em series-parallel graphs} (see~\cite{D65}). Note that, being $K_4$-minor-free, partial $2$-trees are planar.

The interest in viewing $K_4$-minor-free graphs as partial $2$-trees lies in the following natural definition: given a $2$-tree $G$, a \emph{$2$-tree ordering} of its vertices is an ordering $v_1, v_2 \ldots, v_n$ where $v_1v_2$ is an edge and each $v_i$ (for $i\geq 3$) has exactly two neighbors among $\{v_1,\ldots,v_{i-1}\}$; moreover, these two neighbors are themselves adjacent. Observe that in general, more than one such ordering can be associated to a $2$-tree. In this work, a vertex ordering of a $2$-tree will always refer to a $2$-tree ordering. Then, an ordering of a partial $2$-tree will be an ordering of one of its completions to a $2$-tree (such an ordering may contain vertices that are not in the partial $2$-tree). Observe again that there can be more than one such completion.

\subsection{Signed graphs}

We now provide the necessary terminology about signed graphs. A {\em signed graph} is a graph $G$ whose edges are assigned signs: $+$ or $-$. We consider both these signs as the elements of the $2$-element (multiplicative) group. By letting $\Sigma$ be the set of negative edges, we write $(G,\Sigma)$ to denote this signed graph. The set $\Sigma$ will be referred to as the {\em signature of $(G,\Sigma)$}.

A \emph{resigning} of a signed graph $(G,\Sigma)$ is a switch of the signs of all edges along a cut of $G$. This can also be seen as a sequence of resignings where at each step, the signs of all edges incident to a single vertex are switched. Two signatures on a graph $G$ are said to be \emph{equivalent} if one is obtained from the other by a resigning. It is easy to verify that this is an equivalence relation on the set of all signatures of $G$.

Given a signed graph $(G, \Sigma)$, a closed walk $C$ of $G$ is said to be \emph{balanced} (respectively, \emph{unbalanced}) if the product of all edges of $C$ (repetitions included) is $+$ (respectively $-$). Observe that the balance of a closed walk is invariant under resigning. The signature $\Sigma$ partitions the set of cycles of $G$ into two classes of cycles: balanced and unbalanced. The following theorem of Zaslavsky claims that this partition determines the equivalence class of signatures to which $\Sigma$ belongs.

\begin{theorem}
  \label{thm:Zaslavsky}\cite{Z82}
  Two signatures $\Sigma_1$ and $\Sigma_2$ on a graph $G$ are equivalent if and only if they induce the same set of unbalanced cycles.
\end{theorem}

Loops and \emph{unbalanced digons} (unbalanced cycles of length~$2$) play an important role in the study of coloring and homomorphisms of signed graphs (see the article by Brewster et~al.~\cite{BFHN17}). As previously mentioned, we mostly consider simple graphs and will specifically use the term multigraph when multiple edges are allowed. This restriction allows a smoother definition of signed graph homomorphisms.

\paragraph{Homomorphisms of signed graphs}
Given two signed graphs $(G,\Sigma)$ and $(H, \Pi)$, a mapping $\vphi: V(G) \to V(H)$ is said to be a \emph{homomorphism} of $(G,\Sigma)$ to $(H, \Pi)$ if the image of each closed walk is a closed walk with same balance. When such a homomorphism exists, we may write $(G,\Sigma) \to (H, \Pi)$. Given a class $\mathcal{C}$ of signed graphs we say $(H, \Pi)$ \emph{bounds} $\mathcal{C}$ if every element of $\mathcal{C}$ maps to $(H,\Pi)$. Observe that if $uv$ is an edge of $G$, then the natural closed $2$-walk $u-v-u$ is always balanced (independently of the sign of $uv$), thus its image must also be a balanced closed $2$-walk, that is, an edge of $H$. Thus, a homomorphism of $(G,\Sigma)$ to $(H, \Pi)$ is always a homomorphism of $G$ to $H$. In practice, we can use the following theorem to determine whether there exists a homomorphism from $(G,\Sigma)$ to $(H, \Pi)$. The proof of this theorem is a simple application of Theorem~\ref{thm:Zaslavsky}.

\begin{theorem}\label{thm:EquivalentDefinitionOfHom} Given two signed graphs $(G,\Sigma)$ and $(H, \Pi)$, a mapping $\vphi: V(G) \to V(H)$ is a homomorphism of $(G,\Sigma)$ to $(H, \Pi)$ if there exists a signature $\Sigma'$ of $G$ equivalent to $\Sigma$ such that:
\begin{itemize}
\item[(i)] for every edge $uv$ of $G$, $\vphi(u)\vphi(v)$ is an edge   of $H$, and
\item[(ii)] $uv$ is in $\Sigma'$ if and only if $\vphi(u)\vphi(v)$ is   in $\Pi$.
\end{itemize}
\end{theorem}

In a sense, a homomorphism from a signed graph to another is a homomorphism between two equivalence classes of signed graphs. The following lemma provides an easy ``no-homomorphism'' condition. Let $UC_r$ be the signed graph $(C_r, \Sigma)$ where $\Sigma$ is a single edge (it is a representative for all cycles of order $r$ having an odd number of negative edges).

\begin{lemma}\label{lem:MappingUnbalancedCycle}
  Given two positive integers $k$ and $l$, $UC_l\to UC_k$ if and only if $l$ and $k$ have the same parity and $l\geq k$.
\end{lemma}

This leads to the notion of \emph{unbalanced-girth} of a signed graph, that is, the length of a shortest unbalanced cycle of the given signed graph (if no such cycle exists, we define it to be infinite). In fact, Lemma~\ref{lem:MappingUnbalancedCycle} suggests to consider two variants of unbalanced-girth: the length of a shortest unbalanced even-cycle, and the length of a shortest unbalanced odd-cycle. However, in this work we are only concerned with signed bipartite graphs, thus we only have one kind of unbalanced cycles: the even-length ones. It follows that if a signed bipartite graph $(G, \Sigma)$ admits a homomorphism to a signed bipartite graph $(B,\Pi)$, then the unbalanced-girth of $(G, \Sigma)$ is at least the unbalanced-girth of $(B,\Pi)$.

In this work, the following question is our main focus.

\begin{problem}
Given a signed bipartite graph $(B, \Pi)$ of unbalanced-girth $g$, does every signed bipartite $K_4$-minor-free graph of unbalanced-girth at least $g$ admit a homomorphism to $(B, \Pi)$?
\end{problem}

We shall provide a sufficient and necessary condition on $(B, \Pi)$ to be such a bound. As a corollary, we prove that signed projective cubes are such bounds. We now introduce this family of graphs.

\paragraph{Projective cubes}
The \emph{projective cube} of dimension $d$, denoted $PC(d)$, is the Cayley graph $(\mathbb{Z}_2^{d}, S=\{e_1,e_2, \ldots, e_{d},J\})$ where $\{e_1,e_2, \ldots, e_{d}\}$ is the standard basis and $J$ is the all-$1$ vector. Each edge of $PC(d)$ can be associated with one of these vectors (which is equal to the difference between the edge's endpoints). For $d=1$, as $e_1=J$, we allow a multi-edge, and thus $PC(1)$ is a complete graph on two vertices with two parallel edges joining these two vertices. For $d\geq 2$, $PC(d)$ is a simple graph; $PC(2)$ is $K_4$, $PC(3)$ is $K_{4,4}$ and $PC(4)$ is
the Clebsch graph.

\paragraph{Signed projective cubes}
The \emph{signed projective cube} of dimension $d$, denoted $SPC(d)$, is the signed graph $(PC(d), \mathcal{J})$ where $\mathcal{J}$ is the set of edges associated to the vector $J$. Thus, $SPC(1)$ is the unbalanced digon. Further, $SPC(2)$ is the signed $K_4$ where two parallel (non-incident) edges are negative. The graph $SPC(3)$ is the signed $K_{4,4}$ where the signature consists of the edges of a perfect matching. The following theorem by Naserasr et al.~\cite{NRS13} collects some of the key properties of the signed projective cubes $SPC(d)$.

\begin{theorem}\label{thm:SPC(d)}\cite{NRS13}
  The signed projective cube $SPC(d)$ satisfies the following:
  \begin{itemize}
  \item It is of unbalanced-girth $d+1$.
  \item In every unbalanced (respectively balanced) cycle $C$ of $SPC(d)$ and for each element $x$ of $\{e_1, e_2, \ldots, e_d, J\}$, an odd (respectively even) number of edges of $C$ correspond to $x$.
  \item Each pair of vertices belong to a common unbalanced cycle of length $d+1$.
  \end{itemize}
\end{theorem}
Furthermore, it can be easily verified that when $d$ is an odd number,
the underlying graph $PC(d)$ of $SPC(d)$ is a bipartite graph. Thus,
$SPC(2d-1)$ is a signed graph in which all cycles, balanced or
unbalanced, are even. While for $d \geq 2$ the girth of the graph is
always~$4$, the unbalanced-girth of $SPC(2d-1)$ is $2d$.

\subsection{Signed edge-weighted graphs and algebraic distance}
The last general notion to be introduced is the one of signed edge-weighted graphs, an important tool for our work. Let $\mathbb{Z}^*$ be the set of nonzero integers. A \emph{signed edge-weighted  graph} $(G,w)$ is a graph $G$ together with an edge-weight function $w:E(G) \to \mathbb{Z}^*$. Given a signed edge-weighted graph $(G,w)$, a resigning, as before, is to multiply the weights of all edges in an edge-cut by $-1$. Given an edge $e$ of a signed edge-weighted graph the \emph{absolute weight} of $e$ is $|w(e)|$ and the sign of $e$ is positive if $w(e)>0$, and negative otherwise. 

A signed edge-weighted graph is said to be \emph{bipartite} if there is a partition $\{X,Y\}$ of its vertices such that (i) for each edge $e$ with one endpoint in $X$ and the other one in $Y$, the weight $w(e)$ is an odd number (ii) for each edge $e$ with both endpoints in the same part, the weight $w(e)$ is an even number.

This definition is a natural extension of the notion of bipartiteness when weights of edges are defined by the distances in a graph. To be more precise, if $G$ is a connected bipartite graph, then the distance between a pair of vertices is odd if they are in different parts and even if they are in the same part. As we will work with this extended notion, we will repeatedly refer to a \emph{triangle} in a bipartite signed edge-weighted graph. A defining character of these graphs is that the total \emph{weight} of the edges of any given cycle is even (but not necessarily the \emph{number} of these edges).

Given two signed edge-weighted graphs $(G,w_1)$ and $(H,w_2)$, we say that
there is a homomorphism of $(G,w_1)$ to $(H,w_2)$ if, after a resigning
$w'_1$ of $w_1$, there is a homomorphism of $G$ to $H$ which preserves
the edge-weights. Note that if each weight is either $+1$ or $-1$,
then we have the notion of signed graphs and homomorphisms of signed graphs; if all
edges have the same weight, we have the classic notion of graph
homomorphisms. Furthermore, if $(G,w_1) \to (H,w_2)$ where $(H, w_2)$ is a bipartite signed edge-weighted graph, then $(G, w_1)$ must also be a bipartite signed edge-weighted graph.

Given a graph $G$ and a pair $u$ and $v$ of vertices, we denote by
$d_G(u,v)$ the distance in $G$ of $u$ to $v$. For a signed graph
$(G, \Sigma)$, we define the \emph{algebraic distance} between two
vertices $u$ and $v$ as follows:
\begin{equation*}
  ad_{(G,\Sigma)}(u,v) = \begin{cases} d_G(u,v) & \mbox{if } \text{ there is a balanced } u-v
    \text{ path of length } d_G(u,v)  \\
    -d_G(u,v) & \text{otherwise.}\end{cases}
\end{equation*}

While the absolute value of $ad_{(G,\Sigma)}(u,v)$ is only a
function of $G$, its sign is a function of $\Sigma$ and it may switch
if we resign at exactly one of $u$ and $v$.
If the graph $G$ is bipartite, then, with respect to any signature $\Sigma$, the signed edge-weighted graph $(G,ad_{(G,\Sigma)})$ is also bipartite.

For a connected signed graph $(G, \Sigma)$ of order $n$, the
\emph{complete signed distance graph} $(K_n,ad_{(G,\Sigma)})$ of $(G, \Sigma)$ is the  signed edge-weighted complete graph on vertex set $V(G)$, where for
each edge $uv$ of $K_n$, its weight is the algebraic distance
$ad_{(G,\Sigma)}(u,v)$. Given any spanning subgraph $H$ of $K_n$, the
\emph{partial signed distance graph} $(H,ad_{(G,\Sigma)})$ of $G$ is the spanning
subgraph of $(K_n, ad_{(G,\Sigma)})$ whose edges are the edges of
$H$. Furthermore, if for every edge $xy$ of $H$ we have
$|ad_{(G,\Sigma)}(x,y)|\leq k$, we say that $(H,ad_G)$ is a \emph{$k$-partial signed distance graph of $(G,\Sigma)$}.

\begin{observation}\label{obs:DistanceDetmByCycle}
 Let $(G,\Sigma)$ be a signed graph of unbalanced-girth $l$ and let
 $C$ be an unbalanced cycle of length $l$ in $(G, \Sigma)$. Then, for
 any pair $\{u,v\}$ of vertices of $C$, the (algebraic) distance in $G$
 between $u$ and $v$ is determined by their (algebraic) distance in
 $C$ (with respect to the signs induced by $\Sigma$).
\end{observation}
\begin{proof}
 We first observe that the classic distance of $u$ to $v$ is determined by their distance in $C$. Suppose, to the contrary, that there exists a $u-v$ path $P$ in $G$ shorter than the shortest $u-v$ path on $C$. Then, the closed walk starting at $u$, walking through $P$ then $C$ and then $P$ again, can be decomposed into at least two cycles, each of which is smaller than $C$ and at least one of which is unbalanced, contradicting the fact that the length of $C$ corresponds to the unbalanced-girth of $(G,\Sigma)$. If $P$ is of the same length as the distance in $C$ of $u$ and $v$, but not of the same algebraic sign, then together with the shortest $u-v$ path in $C$, they form an unbalanced closed walk of length smaller than that of $C$. This closed walk must contain an unbalanced cycle of length smaller than $C$, unless $C$ is an even cycle and $P$ is of length exactly half of $C$, in which case we have defined the algebraic distance to be the positive one.
\end{proof}

The following fact follows from Observation~\ref{obs:DistanceDetmByCycle}.

\begin{observation}\label{obs:HomDistanceDetmByCycle}
 Suppose that $(G,\Sigma)$ and $(H, \Pi)$ are two signed graphs of unbalanced-girth $l$ and that $\vphi$ is a homomorphism of $(G, \Sigma)$ to $(H, \Pi)$.
 If $u$ and $v$ are two vertices of $(G, \Sigma)$ on a common unbalanced cycle of length $l$ of $(G, \Sigma)$, then $d_H(\vphi(u), \vphi(v))=d_G(u,v)$.
\end{observation}

In Observation~\ref{obs:HomDistanceDetmByCycle}, the distances are considered as the distance functions
in the underlying graphs, thus they are always positive. However, by carefully
resigning $(G,\Sigma)$ and with the homomorphism described in
Theorem~\ref{thm:EquivalentDefinitionOfHom}, we may extend this observation to
algebraic distances. In other words, if $(G,\Sigma)$ and $(H, \Pi)$ are
two signed graphs of unbalanced-girth $l$ and $\vphi$ is a homomorphism
of $(G, \Sigma)$ to $(H, \Pi)$ which, furthermore, preserves the signs
under $\Sigma$ and $\Pi$, then for any pair $u$ and $v$ of vertices of
$G$ on an unbalanced $l$-cycle of $(G,\Sigma)$, we have
$ad_{(G,\Sigma)}(u,v)=ad_{(H, \Pi)}(\vphi(u), \vphi(v))$.

\section{Main property: the notion of $k$-good triples}\label{sec:mainprop}

We now present a key property used in this work. We first fix a positive
integer $k$ for this whole section, and define a class of bipartite
signed graphs.

\paragraph{The signed graph $T_{2k}(p,q,r)$}
Let $p,q$ and $r$ be three non-zero integers each having absolute
value at most $k$ and such that their sum ($p+q+r$) is even. The signed bipartite graph $T_{2k}(p,q,r)$ is any signed graph built as follows. Let
$u,v$ and $w$ be three vertices, join $u$ and $v$ by two disjoint
paths: one of length $|p|$ and the other of length $2k-|p|$. Do the
same between $u$ and $w$ (paths of lengths $|q|$ and $2k-|q|$) and
between $v$ and $w$ (paths of lengths $|r|$ and $2k -
|r|$). Concerning signs, if $p$ is negative, pick one edge on the path
of length $|p|$ between $u$ and $v$ to be negative; if $p$ is
positive, pick one edge on the path of length $2k-|p|$ between $u$ and
$v$ to be negative. Act accordingly for the two paths connecting $u$ and $w$ with respect to $q$ and for the two paths connecting $v$ and $w$ with respect to $r$.

Observe that the underlying graph of $T_{2k}(p,q,r)$ is uniquely determined by $p,q,r$ and $k$. While there is some freedom of choice for the signature of $T_{2k}(p,q,r)$, it is easily verified that all such signatures yield equivalent signed graphs. Thus, in our work, which of these signatures is selected is of no importance. We should also note that if we resign at $u$, or at $v$ or at $w$, we obtain signatures representing signed graphs $T_{2k}(-p,-q,r)$, $T_{2k}(p,-q,-r)$ and $T_{2k}(-p,q,-r)$, respectively. Thus, all these graphs are equivalent up to resigning. Observe that since $p+q+r$ is assumed to be even, all cycles of this graph are of even length, thus the underlying graph is always bipartite.

\subsection{The $k$-good triples}
Let $(p,q,r)$ be a triple of non-zero integers whose absolute values are at most $k$ and whose sum is even. We say $(p,q,r)$ is a {\em $k$-good} triple if $T_{2k}(p,q,r)$ has unbalanced-girth $2k$. Thus, the conditions on the triple $(p,q,r)$ are implicit when we say it is $k$-good. Note that there are a total of $11$ cycles in $T_{2k}(p,q,r)$. Three cycles using exactly two of $u,v,w$ each (which are unbalanced cycles of length $2k$ by the construction), and eight cycles using all three of these vertices (of which exactly four are unbalanced). Thus, checking if a triple is $k$-good is an easy task, but we simplify this
even further in the following proposition.

Three positive numbers $a,b,c$ satisfy the {\em triangle inequality}
if $a\leq b+c$, $b\leq a+c$ and $ c\leq a+b$.

\begin{proposition}\label{prop:test k-good}
  Given integers $p,q,r$ such that $p+q+r$ is even and $1\leq |p|, |q|,
  |r| \leq k$, the triple $(p,q,r)$ is $k$-good if and only if either
  \begin{itemize}
  \item $pqr<0$ and $|p|+|q|+|r| \geq 2k$, or
  \item $pqr>0$ and the triple $\{|p|,|q|, |r|\}$ satisfies the
    triangle inequality.
  \end{itemize}
\end{proposition}

\begin{proof}
  Since $ |p|, |q|, |r| \leq k$, if $pqr<0$, then the shortest
  unbalanced cycle of $T_{2k}(p,q,r)$ using all three of $u$, $v$ and $w$ is of length $|p|+|q|+|r|\geq 2k$. If
  $pqr>0$, then, without loss of generality, we may assume $|p|\leq |q|\leq |r|$. Now,
  the shortest unbalanced cycle of $T_{2k}(p,q,r)$ using all three of $u$, $v$ and $w$ is of length
  $|p|+|q|+2k-|r|\geq 2k$. The claim of the proposition then follows
  easily.
\end{proof}

In particular, we have the following sufficient conditions for a triple to be $k$-good.

\begin{observation}\label{obs:special_good_triple}
  Given three nonzero integers $p,q,r$ such that $p+q+r$ is even and $1\leq
  |p|, |q|, |r| \leq k$,
 \begin{itemize}
 \item[(i)] if two of $|p|,|q|,|r|$ belong to $\{k-1, k\}$, then $(p,q,r)$ is a $k$-good triple.
 \item[(ii)] if one of $|p|,|q|,|r|$ is equal to $k$ and $\{|p|,
   |q|,|r|\}$ satisfies the triangle inequality, then $(p,q,r)$ is a
   $k$-good triple.
 \end{itemize}
\end{observation}

\begin{proof}
  For the first claim, if two of $|p|,|q|,|r|$ belong to $\{k-1, k\}$, since the third is nonzero, then the triangle inequality always holds. Furthermore, if they are both $k-1$, then as $p+q+r$ is an even number, the third is at least~$2$ in absolute value. Thus, we always have $|p|+|q|+|r|\geq 2k$. Hence, in all cases, Proposition~\ref{prop:test k-good} applies.

  For the second claim, we may also use the same proposition because the triangle inequality holds and consequently $|p|+|q|+|r|\geq 2k$.
\end{proof}

\subsection{The all $k$-good property}

Now we are ready to define the main property we will use in this work. Let $(G,w)$ be a bipartite weighted signed graph such that $|w(e)|\leq k$ for any edge $e$ of $G$. We say that $(G,w)$ has the {\em all $k$-good property} if for any edge $uv$ of $G$ with weight $p$ and any integers $q$ and $r$ such that $(p,q,r)$ is a $k$-good triple, there exists a vertex $z$ in $V(G)$ such that $uz$ and $vz$ are edges of $G$ with either $w(uz)=q$ and $w(vz)=r$ or $w(uz)=-q$ and $w(vz)=-r$.

The first consequence of this definition is that if $G$ has at least one edge and $w$ is a weighting such that $(G,w)$ is bipartite and has the all $k$-good property, then for each $k$-good triple $(p,q,r)$, there is a triangle in $(G,w)$ whose edge weights are either $p$, $q$ and $r$ or equivalent to them (that is, $(p,-q,-r)$, etc.). More formally, we have the following lemma.

\begin{lemma}\label{k-goodImpliesAll(pqr)}
  Let $G$ be a graph with a nonempty set of edges and $w$ a weight
  function on its edges such that $|w(e)|\leq k$ for any edge $e$
  of $G$ and that $(G,w)$ is a bipartite weighted graph. If $(G,w)$ has the all $k$-good property, then for each
  $k$-good triple $(p,q,r)$, one of the triples $(p,q,r)$, $(-p,-q,r)$,
  $(-p,q,-r)$, $(p,-q,-r)$ corresponds to the weights on the edges of
  a triangle in $(G,w)$.
\end{lemma}

\begin{proof}
  Since $(G,w)$ has the all $k$-good property, it is enough to prove
  that $G$ has an edge of weight $p$ or $-p$ for every positive integer $p$ satisfying $p\leq k$. As $G$ is not empty,
  there is an edge $e$ with weight $w(e)$. Considering the parity of $w(e)$, choose $a$ from $\{k-1,k\}$ such that $a+k+w(e)$ is even.
  By Observation~\ref{obs:special_good_triple}, $(w(e),k,a)$ is a
  $k$-good triple. Since $(G,w)$ has the all $k$-good property, $e$ is part of a triangle of $G$ with
  weights equivalent to $(w(e),k,a)$. This means that there is an edge $e'$
  in $(G,w)$ of weight $k$ or $-k$. Once again,
  Observation~\ref{obs:special_good_triple} tells us that for a choice of $b\in\{k,k-1\}$ (depending only on the parity of $p$), the triple $(k,b,p)$
  is a $k$-good triple. Hence, $e'$ must be in a triangle with
  weights equivalent to $(k,b,p)$. This means that there is an edge in $(G,w)$ with
  weight $p$ or $-p$, completing the proof.
\end{proof}

\subsection{Signed projective cubes and the all $k$-good property}

In this subsection, we give an example of a graph that satisfies the
all $k$-good property. We present it as a theorem since the next section
will make it a bound for the class of graphs that we consider. The
proof is slightly technical.

\begin{theorem}\label{thm:SPC all k-good}
 The complete signed distance graph of $SPC(2k-1)$, $k\geq 2$, satisfies the all $k$-good property.
\end{theorem}

\begin{proof}
  Recall that the algebraic distance in a signed graph is the distance in the underlying graph equipped with a sign: a positive sign when at least one shortest path has an even number of negative edges, and a negative sign when all shortest paths have of an odd number of negative edges.
  
By Theorem~\ref{thm:SPC(d)}, any two vertices of $SPC(2k-1)$ are in a
common unbalanced cycle of length $2k$. 
Thus, the algebraic distances are bounded above
by $k$ in absolute value. Now, let $x$ and $y$ be two vertices in
$SPC(2k-1)$ with algebraic distance $p$ between them. Since
$SPC(2k-1)$ is vertex-transitive, we may assume that $x$ is the all-$0$
vector. Being at distance $p$ means that vectors disagree on exactly
$p$ coordinates if $p$ is positive (in which case $y$ is obtained from $x$ by adding $p$ distinct elements of $\{e_1,e_2, \ldots, e_{2k-1}\}$), or agree on exactly $|p|-1$
coordinates if $p$ is negative (when $y$ is obtained from $x$ by adding $J$ and $p-1$ distinct elements of $\{e_1,e_2, \ldots, e_{2k-1}\}$). We distinguish between cases when $p$
is positive or negative.

\paragraph{When $p$ is positive}
In this case, we may assume without loss of generality that $y$ is the
vector with the first $p$ coordinates equal to~$1$ and the remaining
$2k-1-p$ coordinates equal to~$0$. Let $q$ and $r$ be two integers such
that $(p,q,r)$ is $k$-good.
\begin{itemize}
\item If both $q$ and $r$ are positive, we let $t$ be the number
  $\frac{p + q - r}{2}$. Since $(p,q,r)$ is $k$-good, the sum $p+q+r$
  is even, and so must be $p+q-r$; thus $t$ is an integer. Moreover, by
  Proposition~\ref{prop:test k-good}, since $pqr > 0$, $p,q$ and $r$
  satisfy the triangle inequality. This ensures that $t$ is
  non-negative. For the same reason, we have $q -r\leq p$ and $p-r\leq q$; the former implies that $t
  \leq p$ and the latter implies that $t\leq q$. Thus, $t$ is smaller than both $p$ and $q$.
  Next, we claim that $q-t\geq 2k-1-p$. By replacing the value of $t$, this is equivalent to the claim that
  $\frac{p+q+r}{2}\leq  2k-1$, which is the case as each of $p$, $q$ and $r$ are at most $k$ and $k\geq 2$.

 Thus, to complete the proof in this case, we choose a vector $z$ with $q$ coordinates equal to~$1$,
  where $t$ of them are in the support of $y$. Then, to get $z$ from $x$, one may add to it $q$ distinct elements of $\{e_1,e_2, \ldots, e_{2k-1}\}$; since $q\leq k$, the algebraic distance of $y$ to $x$ is $q$. On the other hand, to get $z$ from $y$, one may add $p-t$ elements of $\{e_1,e_2, \ldots, e_{2k-1}\}$ which agree on the support of $y$, and then add $q-t$ more elements to match the remaining coordinates. Thus, we added a total of $p+q-2t=r$ elements, and therefore the algebraic distance of $y$ to $z$ is $r$.

\item If both $q$ and $r$ are negative, we consider their opposites
  and find a triangle with distances $p, -q$ and $-r$ by the previous argument.

\item If $qr$ is negative, we may assume, without loss of generality, that $q$ is negative and $r$ is positive. We let $t$ be the number
  $\frac{p+|q|+r-2k}{2}$. By the same parity argument as earlier, this number is
  an integer. Since $pqr < 0$, Proposition~\ref{prop:test k-good}
  ensures that $p + |q| + r \geq 2k$, and hence $t$ is non-negative. As $|q|$ and $r$ are bounded above by $k$, $t$ cannot exceed $p$ (or $\frac{p}{2}$ to be more precise). Similarly, $t$ is bounded above by $\frac{|q|}{2}$ with equality only when $p+r=2k$, in which case, because of parity, $|q|\geq 2$; thus in all cases $t\leq |q|-1$.
  Furthermore, we have $|q|-1-t\leq 2k-p$, that is to say $p+|q|-1-t\leq 2k$ which is indeed the case as we have assumed $|p|,|q|\leq k$.
  Therefore, we can pick $z$ to be a vector with $|q|-1$ coordinates equal
  to~$0$, with $t$ of them being in the support of $y$.
  To get $z$ from $x$ (with a shortest path), one may first add $J$ and then $|q|-1$ elements of $\{e_1,e_2, \ldots, e_{2k-1}\}$ corresponding to the $0$-coordinates of $z$. Thus, $z$ is at algebraic distance $q$ from $x$. To get $z$ from $y$ via a shortest path, we may add $t$ elements of  $\{e_1,e_2, \ldots, e_{2k-1}\}$ to agree on the support coordinates of $y$, and then $2k-1-p-(|q|-1-t)$ elements of $\{e_1,e_2, \ldots, e_{2k-1}\}$ to agree on the remaining coordinates, a total of $2t+2k-p-|q|=r$ elements. Thus, $z$ is at algebraic distance $r$ from $y$.
\end{itemize}

\paragraph{When $p$ is negative}
In this case, we may assume, without loss of generality, that $y$ is the
vector with the first $|p|-1$ coordinates equal to~$0$ and the remaining
$2k-|p|$ coordinates equal to~$1$. Let $q$ and $r$ be two integers such
that $(p,q,r)$ is $k$-good.
\begin{itemize}
\item If both $q$ and $r$ are positive, we let $t$ be the number
  $\frac{|p|+q+r-2k}{2}$. Once again, this is an integer by a parity
  argument. Moreover it is non-negative because $pqr < 0$, and thus by Proposition~\ref{prop:test k-good} $|p|+q+r\geq 2k$. If $t=0$, then we choose $u$, $v$ and $w$ to be three vertices of an unbalanced cycle of length $2k$ of $SPC(2k-1)$ such that the $u-v$ path, the $v-w$ path and the $w-u$ path of this cycle are of lengths respectively $|p|$, $q$ and $r$ and such that the negative edge (the one corresponding to $J$) is on the $u-v$ path. Thus, we may assume that $t\geq 1$. This in turn implies that $|p|, q, r \geq 2$, indeed if one of them is $1$, being subject to parity and upper-bounded by $k$, the other two must be $k$ and $k-1$, implying that $t=0$, contradicting $t\geq 1$.
  Next, we claim (under the assumption that $t\geq 1$) that $t$ is bounded above by $|p|-1$ and by $q-1$. This is because $|p|,q,r \leq k$, which in turn implies that $t\leq \frac{|p|}{2}$ and $t\leq \frac{q}{2}$, which are respectively smaller than $|p|-1$ and $q$ as $|p|, q \geq 2$.

  Thus, we can pick $z$ to be the vector with $q$ coordinates set to~$1$ with $t$ of them being among the $0$-coordinates of $y$. The algebraic distance of $z$ to $x$ is then $q$, and the algebraic distance from $y$ can be easily computed as before: it is $r$.

\item If both $q$ and $r$ are negative, we consider their opposites
  and find a triangle with distances $p, -q$ and $-r$ by the previous argument.

\item If $qr$ is negative, without loss of generality we may assume that $q$ is positive and $r$
  is negative. Furthermore, we may assume that $|q|\leq |r|$ as otherwise we could replace $q$ and $r$ with $-q$ and $-r$, respectively.
   We let $t$ be the number $\frac{|p|+q-|r|}{2}$. Again, by a parity argument, $t$ is an integer. Since $pqr>0$, by Proposition~\ref{prop:test k-good}, $\{|p|, |q|, |r|\}$ satisfies the triangle inequality, thus $t$ is nonnegative. The case $p=-1$ can be verified as before, thus we may assume that $|p|\geq 2$. This, together with the assumption that $q\leq |r|$, implies that $t\leq \frac{|p|}{2}\leq |p|-1$.
  We then pick $z$ to be the vector with $q$ coordinates equal to~$1$ with $t$ of them
  being among the $0$-coordinates of $y$. This is possible because first, $y$ has $|p|-1$ coordinates equal to~$0$ and $t \leq |p|-1$, and second, we have $t\leq q$ by the same triangle inequality, and finally, the number of coordinates at which $y$ is equal to~$1$ is $2k-|p|$, which is clearly at least  $q-t$ as $|p|,|q|,|r| \leq k$.
  It can then be checked as before that $z$ is at algebraic distance $q$ from $z$ and $r$ from $y$.
\end{itemize}

This concludes the proof.
\end{proof}

\section{Main theorem: characterizing bounds}\label{sec:mainthm}

In this section, we give our main theorem, which is a necessary and sufficient condition for a signed bipartite graph of unbalanced-girth $2k$ to bound the class of signed bipartite $K_4$-minor-free graphs of unbalanced-girth at least $2k$. We will use the notation $\SBSPG{2k}$ to denote this class of signed graphs.

\begin{theorem}\label{thm:MinimalBound}
A signed bipartite graph $(B,\Pi)$ of unbalanced-girth $2k$
bounds $\SBSPG{2k}$ if and only if there exists a nonempty $k$-partial
signed distance graph $(\widetilde{B},ad_{(B, \Pi)})$ of
$(B,\Pi)$ which has the all $k$-good property.
\end{theorem}
\begin{proof}
  We first prove the sufficiency. Let $(B,\Pi)$ be a signed bipartite
  graph of unbalanced-girth $2k$ and let $(\widetilde{B},ad_{(B, \Pi)})$ be a $k$-partial
  signed distance graph of $(B,\Pi)$ which has the all $k$-good property.
  We want to prove that any signed graph in $\SBSPG{2k}$
  admits a homomorphism to $(B,\Pi)$. Let $(G,\Sigma)$ be such a
  signed graph. We may assume that $G$ is connected, as otherwise we
  may apply our technique on each of its connected components.

  The graph $G$ is a partial $2$-tree, so there exists a $2$-tree
  completion $H$ of $G$ ($V(H) = V(G)$ and $E(G) \subseteq E(H)$). Let
  $v_1,v_2,\ldots,v_n$ be a $2$-tree ordering of the vertices of
  $H$. We let $H_i$ denote the subgraph of $H$ induced by
  $v_1,v_2,\ldots,v_i$. This means that $v_i$ is of degree~$2$ in
  $H_i$ and its neighbors form an edge in $H_{i-1}$. We define a
  weight function $\omega$ on the edges of $H$ as follows.
  \begin{equation*}
    \omega(x,y) = \begin{cases}
      ad_{(G,\Sigma)}(x,y) & \mbox{if } |ad_{(G,\Sigma)}(x,y)|\leq k,\\
      k & \mbox{if } d_G(x,y) > k \mbox{ and } d_G(x,y) \mbox{ has same parity as } k,\\
      k-1 & \mbox{otherwise}.
    \end{cases}
  \end{equation*}
  In other words, for vertices far apart, we replace their algebraic
  distance with $k$ or $k-1$ depending on whether they are in the same part of the bipartition of $G$; the resulting weighted graph remains bipartite with our extended terminology.
  Since $k$ is at least~$2$, the function $\omega$ never assigns value zero.

  We first prove that for any three pairwise adjacent vertices $x,y$
  and $z$ of $H$, $(\omega(x,y), \omega(x,z), \omega(y,z))$ is a
  $k$-good triple. All three numbers are integers with absolute values
  between $1$ and $k$, moreover their sum is even as the weighted graph $(H, w)$ is bipartite. If the absolute values of (at least) two of the three weights are in $\{k,k-1\}$, then we are done by Observation~\ref{obs:special_good_triple}.
  If exactly one of the absolute values is in $\{k,k-1\}$, then, by the triangle inequality and the parity argument, both conditions of Proposition~\ref{prop:test k-good} are met, unless one of the values is $k-1$ and the sum of the other two is $k-1$. In this case, we claim that $\omega(x,y) \omega(x,z) \omega(y,z)\geq 0$; then, the second part of Proposition~\ref{prop:test k-good} applies. To see that the latter claim holds, without loss of generality, suppose that $|\omega(x,y)|=k-1$ and $|\omega(x,z)|+|\omega(y,z)|=k-1$; then, by the definition of $\omega$, we know that $d_G(x,z)=|\omega(x,z)|$, $d_G(y,z)=|\omega(y,z)|$, and by the triangle inequality of the distance function in $G$, we have that $d_G(x,y)=k-1$. Now, if $\omega(x,y) \omega(x,y) \omega(y,z)\geq 0$, the closed walk starting at $x$ and passing through $y$ and $z$ using paths corresponding to $\omega(x,y)$, $\omega(x,y)$, and $\omega(y,z)$ is an unbalanced walk of length $2k-2$, which must contain an unbalanced cycle of length at most $2k-2$. This contradicts our assumption on the unbalanced-girth of $(G, \Sigma)$.

  We can now build a homomorphism $\vphi$ from $(H,\omega)$ to
  $(\widetilde{B},ad_{(B, \Pi)})$. We shall build this homomorphism
  iteratively by finding an appropriate image for each vertex in the
  $2$-tree ordering of $V(H)$. The weighted graph $(H_3,\omega)$ is a
  triangle, thus the triple of its edge-weights is a $k$-good
  triple. By Lemma~\ref{k-goodImpliesAll(pqr)}, there is an equivalent
  triangle in $(\widetilde{B},ad_{(B, \Pi)})$. We may resign at
  $v_1,v_2$ or $v_3$, if needed, and map those vertices to the adequate triangle
  (we recall that a homomorphism allows resigning in the source
  graph). We may repeat this process. For $i \geq 3$, vertex $v_{i+1}$
  has two neighbors in the graph $H_{i+1}$. Those two neighbors (say
  $x$ and $y$) are already mapped to $(\widetilde{B},ad_{(B,
    \Pi)})$. Now, the three values of $\omega$ on the triangle defined
  by $x,y$ and $v_{i+1}$ form a $k$-good triple. Since
  $(\widetilde{B},ad_{(B, \Pi)})$ has the all $k$-good property, there
  is some vertex that makes an equivalent triangle with $\vphi(x)$ and
  $\vphi(y)$. If needed, we may resign around $v_{i+1}$ in $H$ in
  order to match the exact values on this triangle. In the end, we
  obtain a homomorphism from $(H,\omega)$ to $(\widetilde{B},ad_{(B,
    \Pi)})$. The edges of $G$ are exactly those edges of $H$ which are
  of weight $+1$ or $-1$. Since $\vphi$ assigns these edges to edges of
  $(\widetilde{B},ad_{(B, \Pi)})$ with weight $+1$ or $-1$, it means
  that $\vphi$ is also a homomorphism of $(G, \Sigma)$ to $(B,
  \Pi)$. This concludes the proof of sufficiency.

  \bigskip

  We now prove the necessary part of the statement. First, we assume that $(B, \Pi)$ is a minimal signed bipartite graph of unbalanced-girth $2k$ bounding the class $\SBSPG{2k}$. Our aim is to build a $k$-partial signed distance graph of $(B, \Pi)$ with the all $k$-good property.

  Let $\mathcal C$ be the class of $k$-partial signed distance graphs
  $(\widetilde{G},ad_{(G,\Sigma)})$
  satisfying:\\
  (i) $(G, \Sigma)\in \SBSPG{2k}$;\\
  (ii) $\widetilde{G}$ is a partial $2$-tree;\\
  (iii) $(\widetilde{G},ad_{(G,\Sigma)})$ is a $k$-partial signed distance graph of $(G,\Sigma)$;\\
  (iv) for every edge $uv$ of $\widetilde{G}$, $u$ and $v$ lies on a common unbalanced cycle of
  length $2k$ of $(G, \Sigma)$.
  \medskip

  It is clear that $\mathcal C$ is nonempty, for example if $G$ is an unbalanced cycle of length $2k$, then for any $\widetilde{G}$ which contains $G$ as a spanning subgraph we have $\widetilde{G}\in \mathcal{C}$. Let us emphasize that $\widetilde{G}$ does not necessarily contain $G$ as a subgraph, in fact if $G$ contains an edge $e$ which is in no unbalanced cycle of length $2k$, then $e$ cannot be an edge of $\widetilde{G}$.

  Our aim is to show that if $B^*$ is minimal such that $(B^*,ad_{(B, \Pi)})$ bounds $\mathcal C$, then $(B^*,ad_{(B, \Pi)})$ has the all $k$-good property. We first need to show that such a $B^*$ exists. To this end, we show that $(K_{|V(B)|},ad_{(B, \Pi)})$ is a bound for $\mathcal C$.

  Indeed, since $B$ bounds $\SBSPG{2k}$, there exists a homomorphism $f:(G,\Sigma) \to (B,\Pi)$. If necessary, by change of notation, we may assume that $f$ is indeed a homomorphism with respect to $\Sigma$. We claim that $f$ is also a homomorphism of $(\widetilde{G},ad_{(G,\Sigma)})$ to $(K_{|V(B)|},ad_{(B,\Pi)})$ which preserves weights. Clearly, $f$ preserves edges of weight $+1$ or $-1$. Now, let $uv$ be an edge of weight $p$ in $(\widetilde{G}, ad_{(G,\Sigma)})$, $p\geq 2$. By Property~(iv), $u$ and $v$ lie on a common unbalanced cycle of length $2k$ of $(G, \Sigma)$. By Observation~\ref{obs:DistanceDetmByCycle}, we have $ad_{(G,\Sigma)}(u,v)=ad_{(B,\Pi)}(f(u),f(v))$, that is, $uv$ is mapped to an edge of weight $p$, which shows the claim for $f$.

  Now, consider a minimal graph $B^*$ with $B\subseteq B^*$ such that $(B^*,ad_{(B, \Pi)})$ bounds the class $\mathcal C$. By Property~(iv), every edge of a weighted graph in $\mathcal C$ has weight at most~$k$, therefore, this is also the case for $B^*$. In other words, $(B^*,ad_{(B, \Pi)})$ is a $k$-partial signed distance graph of $(B, \Pi)$. We will show that $(B^*,ad_{(B,\Pi)})$ has the all $k$-good property.

  Clearly, we have $E(B^*)\neq\emptyset$. Therefore, assume by contradiction that for some edge $xy$ with $ad_{(B,\Pi)}(x,y)=p$ and some $k$-good triple $\{p,q,r\}$, there is neither a vertex $z$ in $B^*$ with $xz,yz\in E(B^*)$, $ad_{(B, \Pi)}(x,z)=q$ and $ad_{(B, \Pi)}(y,z)=r$, nor a vertex $z'$ with $xz,yz\in E(B^*)$, $ad_{(B, \Pi)}(x,z)=-q$ and $ad_{(B, \Pi)}(y,z)=-r$. By minimality of $B^*$, there exists a weighted signed graph $(\widetilde{G_{xy}},d_{G_{xy}})$ of $\mathcal C$ such that for any homomorphism $f$ of $(\widetilde{G_{xy}},d_{G_{xy}})$ to $(B^*,ad_{(B, \Pi)})$, there is an edge $ab$ of $\widetilde{G_{xy}}$ of weight $p$, such that $f(a)=x$ and $f(b)=y$.

  We now build a new weighted signed graph from $(\widetilde{G_{xy}},d_{G_{xy}})$ as follows. Let $\widehat{T}$ be a $2$-tree completion of the graph $T=T_{2k}(p,q,r)$ where $\widehat{T}$ contains the triangle $uvw$; edges of this triangle have weights $p$, $q$ and $r$ in $(\widehat{T},d_T)$. Then, for each edge $ab$ of $\widetilde{G_{xy}}$ with $d_{G_{xy}}(a,b)=p$, we add a distinct copy of $(\widehat{T},d_T)$ to $(\widetilde{G_{xy}},d_{G_{xy}})$ by identifying the edge $ab$ with the edge $uv$ of $(\widehat{T},d_T)$ (observe that both have weight $p$). It is clear that the resulting weighted signed graph, that we call $(\widehat{G'_{xy}},d_{G'_{xy}})$, belongs to the class $\mathcal C$. Moreover, $(\widetilde{G_{xy}},d_{G_{xy}})$ is a subgraph of $(\widehat{G'_{xy}},d_{G'_{xy}})$ (indeed the new vertices added in the construction have not altered the distances between original vertices of $G_{xy}$). Thus, there exists a homomorphism $\phi$ of $(\widehat{G'_{xy}},d_{G'_{xy}})$ to $(B^*,ad_{(B,\Pi)})$, whose restriction to the subgraph $(\widetilde{G_{xy}},d_{G_{xy}})$ is also a homomorphism. Therefore, by the choice of $G_{xy}$, at least one pair $a,b$ of vertices of $\widehat{G'_{xy}}$ with $d_{G'_{xy}}(a,b)=p$ is mapped by $\phi$ to $x$ and $y$, respectively. But then, the copy of $\widehat{T_{2k+1}}(p,q,r)$ added to $G$ for this pair forces the existence of the desired triangle on the edge $xy$ of $B^*$, which is a contradiction and completes the argument.

  Finally, to complete the proof, suppose that $(B,\Pi)$ is a signed bipartite graph of unbalanced-girth $2k$ which bounds $\SBSPG{2k}$ but is not minimal. Then, there exists a minimal subgraph $(B', \Pi')$ (where $\Pi'$ is induced by $\Pi$) which bounds $\SBSPG{2k}$. Let $(\widetilde{B'},ad_{(B', \Pi')})$ be the $k$-partial signed distance graph of $(B', \Pi')$ which satisfies the theorem for $(B', \Pi')$. Thus, each edge of $\widetilde{B'}$ is in an unbalanced $2k$-cycle of $(B', \Pi')$ and, therefore, of $(B, \Pi)$. Hence, by Observation~\ref{obs:DistanceDetmByCycle}, $ad_{(B', \Pi')}$ is the restriction of $ad_{(B, \Pi)}$ to the edges of $\widetilde{B'}$ which means that $(\widetilde{B'},ad_{(B', \Pi')})$ is a $k$-partial signed distance graph of $(B, \Pi)$ and thus satisfies the theorem for $(B, \Pi)$ as well.
\end{proof}

\section{Algorithms}\label{sec:algo}

We now show three algorithmic applications of
Theorem~\ref{thm:MinimalBound}. The first one addresses the main
question of deciding whether a given signed bipartite graph of
unbalanced-girth $2k$ bounds $\SBSPG{2k}$. We give a polynomial-time
algorithm for this. Then, we discuss the possible YES and NO cases of
this question. For the NO case we show how a certificate can be
built. For $(B, \Pi)$ being a YES case, we give a polynomial-time
algorithm which does the following: for an input $(G, \Sigma)$ which
is a signed bipartite $K_4$-minor-free graph of unbalanced-girth at
least $2k$, our algorithm builds a homomorphism of $(G, \Sigma)$ to
$(B, \Pi)$.

Before going into further details, we would like to point out that
given a general signed graph $(B, \Pi)$, it can be checked in
polynomial time whether it is a signed bipartite graph of
unbalanced-girth $2k$. To start with, one must decide if $B$ is
bipartite. While many efficient methods can be employed, we consider
the following simple method which has a slightly stronger
conclusion. Given some vertex $x$, we form the set $N_i(x)$ of
vertices at distance $i$ from $x$ (by induction on $i$). At step $i$,
$i\leq n-1$, we check if the set $N_i(x)$ induces a stable set. If
for some $i$ this is not the case, then there exists a closed walk of
length $2i+1$ starting from $x$. This walk is not necessarily a cycle,
as it could consist of a path connecting $x$ to an odd-cycle. However,
the smallest of all such walks, calculated over all vertices of $G$,
corresponds to the length of the shortest odd-cycle of $B$.
Furthermore, $B$ is bipartite if no vertex $x$ is in an odd-cycle.

We can extend this algorithm to decide the unbalanced-girth of $(B,
\Pi)$ where $B$ is a bipartite graph. For a given vertex $x$ of $B$
and for $i\leq n-1$, we build two sets $N_i^+(x)$ and $N_i^-(x)$:
$N_i^+(x)$ (respectively, $N_i^-(x)$) is the set of vertices at
distance $i$ from $x$ with a path of length $i$ from $x$ whose product
of signs is positive (respectively, negative). These two sets can be
built by induction on $i$. If for some $i$ we have $N_i^+(x) \cap
N_i^-(x)\neq \emptyset$, then there exists an unbalanced closed walk
of length $2i$ containing $x$. The unbalanced-girth of $(B, \Pi)$ is
then simply the minimum of the lengths of all such unbalanced walks
taken over all vertices, noting that a minimum unbalanced walk must
necessarily be an unbalanced cycle.

Thus, having verified that the input graph $(B,\Pi)$ is a bipartite
signed graph of unbalanced-girth $2k$, the following algorithm decides
if $(B,\Pi)$ bounds the class $\SBSPG{2k}$.

\begin{algorithm}{Deciding whether a signed bipartite graph of unbalanced-girth $2k$ bounds $\SBSPG{2k}$.}
	\begin{algorithmic}[1]
		\Require{An integer $k$, a signed bipartite graph $(B, \Pi)$ of unbalanced-girth $2k$.}
		\Statex
		\State Compute the set $\mathcal T_k$ of $k$-good triples.
		\State Compute the algebraic distance function $ad_{(B, \Pi)}$ of $(B, \Pi)$.\label{alg:distances}

		%
		\State Let $(\widetilde{B},ad_{(B,\Pi)})$ be the $k$-partial signed distance graph of $B$ obtained from the
		complete signed distance graph of $B$ by removing all edges of weight more than $k$ or less than $-k$.
		\State If $E(\widetilde{B})=\emptyset$ return NO.\label{alg:EdgeSet}
		\For{$e=xy$ in $E(\widetilde{B})$}\label{alg:loop}
		\Let{$p$}{$ad_{(B, \Pi)}(xy)$}
		\For{each $k$-good triple $\{p,q,r\}\in\mathcal T_k$ containing $p$}
		\If{there is no pair $z,z'$ of $V(B)$ with either $$ad_{(B, \Pi)}(xz)=ad_{(B, \Pi)}(yz')=q \quad \text{and} \quad ad_{(B,\Pi)}(yz)=ad_{(B,\Pi)}(xz')=r$$ or $$ad_{(B, \Pi)}(xz)=ad_{(B, \Pi)}(yz')=-q \quad \text{and} \quad ad_{(B,\Pi)}(yz)=ad_{(B,\Pi)}(xz')=-r$$}
		\label{alg:check_edge}
		\Statex\Comment{\emph{(the edge $e$ fails for the all $k$-good property)}}

		\Let{$\widetilde{B}$}{$(\widetilde{B}-uv)$}\label{alg:update}
		\State Go to Step~\ref{alg:EdgeSet} (restart the loop).
		\EndIf
		\EndFor
		\EndFor

		\State\Return YES\Comment{\emph{($\widetilde{B}$ is a certificate)}}

	\end{algorithmic}
	\label{algo:YES-NO}
\end{algorithm}

In the next theorem, we show that the algorithm is correct and we discuss its complexity.

\begin{theorem}
	Algorithm~\ref{algo:YES-NO} checks in time $O(m^2n)=O(n^5)$
    whether a given signed bipartite graph $(B, \Pi)$ of
    unbalanced-girth $2k$ on $n$ vertices and $m$ edges bounds
    $\SBSPG{2k}$.
\end{theorem}
\begin{proof}
	First let us prove that Algorithm~\ref{algo:YES-NO} is
    correct, that is, it returns ``YES'' if and only if $(B, \Pi)$
    is a bound for $\SBSPG{2k}$.

    \begin{itemize}
    \item Assume that $(B, \Pi)$ is a bound. Then, by
        Theorem~\ref{thm:MinimalBound}, there is a $k$-partial
        signed distance graph $(\widetilde{B},ad_{(B, \Pi)})$ of
        $(B,\Pi)$ which has the all $k$-good triple
        property. Furthermore, as mentioned in the proof of the
        theorem, a minimal bound $\widetilde{B}$ can be chosen so that
        each of its edges belongs to an unbalanced $2k$-cycle of
        $B$. Therefore, by Observation~\ref{obs:DistanceDetmByCycle},
        the weight function $ad_{(B, \Pi)}$ restricted to the edges of
        $\widetilde{B}$ would remain the same if the distances are
        calculated in the subgraph of $B$ induced by the edges of
        weight $+1$ or $-1$ of $\widetilde{B}$, or if calculated in an
        intermediate subgraph.  Therefore, no edge of
        $(\widetilde{B},ad_{(B, \Pi)})$ will be deleted in
        Step~\ref{alg:loop} of the algorithm and, by finiteness, YES
        will be the output.

      \item Conversely, if $(B,\Pi)$ is not a bound, then by
        Theorem~\ref{thm:MinimalBound}, no subgraph
        $(\widetilde{B},ad_{(B, \Pi)})$ has the all $k$-good
        property. That is, as long as there is an edge, some edge will
        fail the check of Step~\ref{alg:loop} and finally NO will be
        the output.
	\end{itemize}
	For the running time of the algorithm, having picked an edge
        $e$, one must consider at most $n-2$ triangles containing $e$
        in order to decide if $e$ is passing the test. In each call of
        Step~\ref{alg:loop}, one may check at most $m$ edges thus a
        total of $O(mn)$ checks before recalling the loop. Finally,
        the loop may be recalled at most the total number of edges of
        the complete graph on $V(B)$, thus total of $O(m^2n)$ steps.
\end{proof}

\subsection{Certificate of a NO answer}

Assuming that the algorithm returns NO, one can backtrack it to build a signed bipartite $K_4$-minor-free graph of unbalanced-girth $2k$ which does not map to $(B, \Pi)$. Let $e_1, e_2 \ldots, e_j$ be the edges whose weights were between $-k+1$ and $k$ in the order that they are deleted in Step~\ref{alg:loop}. Suppose $e_i$ is of weight $p_i$ and let ${p_i,q_i,r_i}$ be (one of) the $k$-good triples which did not exist on $e_i$. Starting from the end, we first build a weighted $K_4$-minor-free graph as follows. The starting graph is the weighted triangle of weights $p_j, q_j, r_j$. At step $i$, having already built the graph $G_{i+1}$, we build $G_i$ as follows. For every edge of weight $p_i$ of $G_{i+1}$, build two new triangles on it, a $(p_i,q_i, r_i)$-triangle and a $(p_i,r_i, q_i)$-triangle.
We build the signed bipartite graph $G$ from $G_1$ as follows: an  edge $xy$ of weight $l$ is replaced by two paths of lengths $l$ and $2k-l$. If $l$ is negative, we assign a negative sign to an edge of the path of length $l$, otherwise we assign a negative sign to the path of length $2k-l$. Using Observation~\ref{obs:HomDistanceDetmByCycle}, one can easily verify that mapping $G$ to $B$ would imply a mapping of the weighted graph $G_1$ to $(K_{V(B)},ad_{(B, \Pi)})$, which contradicts the removal of $e_1$ at the first run of Step~\ref{alg:loop}.

We note that this construction may end up having an order which is exponential in the order of $B$. Given a signed bipartite graph of unbalanced-girth $2k$ for which Algorithm~\ref{algo:YES-NO} has returned NO, we do not know if there always exists a signed bipartite $K_4$-minor-free graph of unbalanced-girth $2k$ which does not map to $B$ and whose order is polynomial in the order of $B$.

\subsection{$(B, \Pi)$-coloring algorithm for YES instances}

Assuming that a signed bipartite graph $(B,\Pi)$ of unbalanced-girth $2k$ is a YES instance of our algorithm, the next step is to give a polynomial-time algorithm which, for an input $(G,\Sigma)$ (a signed bipartite $K_4$-minor-free graph of unbalanced-girth at least $2k$) gives a homomorphism to $(B,\Pi)$. The next algorithm does exactly this using the weighted graph $(\widetilde{B},ad_{(B,\Pi)})$ which is produced by Algorithm~\ref{algo:YES-NO}.

\begin{algorithm}{Finding a homomorphism of $(G,\Sigma)$ to $(H,\Pi)$.}

	\begin{algorithmic}[1]
		\Require{$(\widetilde{B},ad_{(B,\Pi)})$, the YES certificate of $(B,\Pi)$ and a signed bipartite $K_4$-minor-free graph $(G,\Sigma)$.}
		\Statex
		\State Complete $G$ to a $2$-tree extension $G'$.
		Compute a $2$-tree ordering $v_1, v_2, \ldots, v_n$ of $G'$.
		\State Compute the signed distance function $ad_{(G, \Sigma)}$ for all edges of $G'$. For edges of weight more than $k$ or less than $-k+1$, replace the weight with $k$.

		\State Find a triangle in $(\widetilde{B},ad_{(B,\Pi)})$ whose edges are of the same weights as the edges of the triangle induced by $v_1$, $v_2$ and $v_3$ (possibly after a resigning in one of $v_1$, $v_2$, $v_3$).
		Apply the resigning; if it was needed in finding the triangle, to $(G', ad_{(G, \Sigma)})$. Then, map $v_1$, $v_2$, $v_3$ by $\phi$ to the vertices of the found triangle to preserve the algebraic distances.

		\For{ $i=4, \ldots n$}
                \State Let $x_i$ and $y_i$ be the two neighbors of $v_i$ in $v_1, v_2, \ldots v_i$.
                \State In $(\widetilde{B},ad_{(B,\Pi)})$, find a vertex $z$ whose algebraic distances from $\phi(x_i)$ and $\phi(y_i)$ is either the same as the distances of $v_i$ from $x_i$ and $y_i$ (respectively) or it is their opposites.
                \State Let $\phi(v_i)=z$.
		\State In the latter case, resign $(G', ad_{(G, \Sigma)})$ at vertex $v_i$.
		\label{alg:mapping}

		\EndFor

	\end{algorithmic}
	\label{algo:FindingHom}
\end{algorithm}

\section{Applications}\label{sec:application}

We now apply the results of the previous sections.

\subsection{Application to projective cubes}

An immediate corollary of Theorem~\ref{thm:SPC all k-good} and Theorem~\ref{thm:MinimalBound} is the following.

\begin{theorem}\label{thm:BoundingSPGbySPC}
Every signed bipartite $K_4$-minor-free graph of unbalanced-girth at least $2k$ admits a homomorphism to $SPC(2k-1)$.
\end{theorem}

\subsection{Application to edge-coloring}

Let $G$ be a $K_4$-minor-free $2k$-regular multigraph such that for any set $X$ of an odd number of vertices, at least $2k$ edges are incident with exactly one vertex of $X$. It is easily observed that this is a necessary condition for $G$ to have edge-chromatic number equal to $2k$. Graphs satisfying this condition are called $2k$-graphs.

In~\cite{S90}, Seymour gave a formula for calculating the edge-chromatic number of all $K_4$-minor-free multigraphs, which implies that for the case of $2k$-regular multigraphs with no $K_4$-minor, this necessary condition (to be $2k$-edge-colorable) is also sufficient. As an application of our work, we give a simple proof of this special case.

\begin{theorem}\label{thm:edge-coloring}
Every $K_4$-minor-free $2k$-graph is $2k$-edge-colorable.
\end{theorem}
\begin{proof}
	Let $G$ be a $K_4$-minor-free $2k$-graph. Thus, $G$ is a planar multigraph and we will consider it as a plane multigraph, that is, with a fixed planar drawing, and thus we can speak of the dual of $G$.
	One can easily check that Tutte's theorem can be applied, and thus $G$ has a perfect matching. Let $M$ be such a matching. Let $G^{D}$ be the dual of $G$ and assign a negative sign to all edges corresponding to $M$, and a positive sign to all other edges. Thus $(G^{D}, M')$ is a signed graph which is of unbalanced-girth $2k$ (this correspond to the fact that $G$ is a $2k$-graph). Furthermore, as $M'$ corresponds to the matching $M$, every facial cycle of $(G^{D}, M')$ is an unbalanced cycle of length $2k$.

	Thus, by Theorem~\ref{thm:BoundingSPGbySPC}, there exists a homomorphism $\psi$ of $(G^{D}, M')$ to $SPC(2k-1)$. In this mapping, all facial cycles are mapped to unbalanced cycles of length $2k$. Edges of any such cycle are in a one-to-one correspondence with $\{e_1,e_2, \ldots, e_{2k-1}, J\}$. Re-assigning them to the edges of $G$ then yields a proper edge-coloring of $G'$.
\end{proof}

\section{Concluding remarks}\label{sec:conclu}

We conclude this paper with some discussions.

\subsection{Relation to related work}

The question of mapping a signed graph $(G, \Sigma)$ to signed projective cubes is of high importance as it captures a packing problem (we refer to~\cite{NRS13} for a proof of the following theorem).

\begin{theorem}[\cite{G05}]
Given a signed graph $(G,\Sigma)$, we have $(G,\Sigma)\to SPC(d)$ if and only if $E(G)$ can be partitioned into $d+1$ sets $\Sigma_1, \Sigma_2, \ldots, \Sigma_{d+1}$ where each $\Sigma_i$ is equivalent to $\Sigma$.
\end{theorem}

Thus, in this work, we have shown that such a partition of the edges is possible for every signed bipartite $K_4$-minor-free graph of unbalanced-girth $d+1$. This is a partial support towards the following general conjecture by Guenin~\cite{G05}. (Note that in~\cite{G05}, Guenin actually makes a stronger conjecture which would imply Conjecture~\ref{conj:PlanarToPC}.)

\begin{conjecture}[\cite{G05}]\label{conj:MappingPlanarToSPC(2k-1)}\label{conj:PlanarToPC}
Every planar signed bipartite graph of unbalanced girth $2k$ admits a homomorphism to $SPC(2k-1)$.
\end{conjecture}

The case $k=2$ of Conjecture~\ref{conj:PlanarToPC} is shown to be stronger than the four-color theorem in~\cite{NRS15}.

Regarding edge-coloring, Seymour, strongly extending Tait's reformulation of the four-color theorem, conjectured that the edge-chromatic number of any planar multigraph is determined by its maximum degree and its fractional edge-chromatic number (which was later proved to be computable in polynomial time).

\begin{conjecture}\label{conj:Seymour}
Let $G$ be a planar multigraph. Then $\chi'(G)=\max \{\Delta(G), \chi'_f(G) \}$.
\end{conjecture}

The restriction of this conjecture to the class of regular multigraphs is connected to homomorphisms to projective cubes, as we saw in Theorem~\ref{thm:edge-coloring} (see also~\cite{NRS13}).
Seymour proved Conjecture~\ref{conj:Seymour} when restricted to the class of $K_4$-free multigraphs in~\cite{S90}. Theorem~\ref{thm:edge-coloring} provides a simple proof for special cases of this result.

\subsection{Remarks on the algorithms}

We would now like to give two reasons why the algorithms given in this work are of importance.

Given a signed graph $(H, \Pi)$, the planar $(H, \Pi)$-coloring problem takes as an input a signed planar graph $(G, \Sigma)$ and outputs YES if there is a homomorphism of $(G, \Sigma)$ to $(H,\Pi)$. The algorithmic complexity of this question is of interest. Restricted to the class of graphs, and by the virtue of the four-color theorem, if $H$ contains a $K_4$, then the answer is always YES. On the other hand, using a cross-over gadget, it is shown in \cite[Theorem~2.2]{GJS76} that for $H$ being $K_3$, this is an NP-complete problem. The results of~\cite{N07} imply that for $H$ being the projective cube of dimension~$4$, there is a duality: a planar graph $G$ maps to $PC(4)$ if and only if $K_3$ does not map to $G$. This results in a polynomial-time algorithm for this case. Conjecture~\ref{conj:MappingPlanarToSPC(2k-1)} proposes a duality theorem for $SPC(2k-1)$, thus implying a polynomial-time algorithm. Duality thus far is the only way of producing polynomial-time examples of this general question. However, considering the difficulty level of Conjecture~\ref{conj:MappingPlanarToSPC(2k-1)}, even determining all signed graphs $(H, \Pi)$ for which a duality theorem with one signed graph (or a finite number of them) exists, seems out of reach. Our algorithms then do this work for the simplest non-trivial cases: given a signed bipartite graph $(H, \Pi)$ of unbalanced-girth $2k$, it decides in polynomial-time if $UC_{2k-2}$ can provide a duality theorem for determining if a signed bipartite $K_4$-minor-free graph $(G,\Sigma)$ maps to $(H,\Pi)$. If $(H, \Pi)$ passes the test, then a signed bipartite $K_4$-minor-free graph $(G, \Sigma)$ maps to $(H,\Pi)$ if and only if $UC_{2k-2}$ does not map to $(G, \Sigma)$.

The second reason to describe the algorithm is to implement it. An implementation using Sage is available at \hyperref{https://www.irif.fr/~reza}{}{}{\texttt{https://www.irif.fr/\~{}reza}}. This is of importance for progress on looking for optimal solutions mentioned in the previous section. As mentioned, some optimal solutions are already known. A natural approach then is to guess a few more cases which may then allow to observe a general pattern. Once a good guess is proposed, the program simply checks if the suggestion works, and then if a general pattern is observed one may employ our main theorem to devise a theoretical proof. The bounds of order $O(k^2)$ given in~\cite{BFN17} for the analogous problem were indeed built in this way.

\subsection{Future work and conclusion}

As a corollary of our work, we have shown that every signed bipartite $K_4$-minor-free graph of unbalanced-girth at least $2k$ admits a homomorphism to the signed projective cube $SPC(2k)$. Observing that this graph has $2^{2k-1}$ vertices, we believe that it is far from being an optimal such bound. We believe that a subgraph of order $O(k^2)$ would work, and we have shown in~\cite{BFN17} that this is the case for an analogous problem concerning graphs (with no signature) and odd-girth. It is shown in~\cite{NRS15} that every signed bipartite $K_4$-minor-free graph (of unbalanced-girth at least~$4$) admits a homomorphism to $(K_{3,3}, M)$ where $M$ is a perfect matching of $K_{3,3}$; in fact the result holds even if one of the negative edges is deleted. It can be easily checked that this is a subgraph of $SPC(3)$.

Smaller bounds in general would imply stronger applications to edge-coloring. We refer to~\cite{BFN17} for such analogous results in the case of (non-signed) graphs.

We would also like to point out that any signed bipartite graph of unbalanced-girth $2k$ which bounds $\SBSPG{2k}$ must be of order at least quadratic in $k$. In an ongoing work, He, Naserasr and Sun have built signed bipartite $K_4$-minor-free graphs of unbalanced-girth $2k$ for which any homomorphic image of unbalanced-girth $2k$ would need $\Omega(k^2)$ vertices.

Conjecture~\ref{conj:PlanarToPC} is believed to be true even for some larger classes of graphs such as the class of $K_5$-minor-free graphs. A notable subclass of $K_5$-minor-free graphs is the class of partial $3$-trees. While we plan to address this in forthcoming works, we point out that a signed bipartite partial $3$-tree of unbalanced-girth $2k$ is built in~\cite{NSS15} for which any homomorphic image that is of unbalanced-girth $2k$ has at least $2^{2k-1}$ vertices. The construction in that work is presented as a planar graph, but it can be easily verified that it is also a partial $3$-tree.

In~\cite{G05}, Guenin in fact proposed a stronger conjecture than Conjecture~\ref{conj:PlanarToPC}. He suggested that the condition of planarity not only can be replaced by being $K_5$-minor-free, but also by \emph{odd-$K_5$-minor-free}. Restricted this to the bipartite case, that would be to say:

\begin{conjecture}[\cite{G05}]\label{conj:BoundingBipartiteK_5free}
Let $(G,\Sigma)$ be a signed bipartite graph of unbalanced-girth at least $2k$ which does not have a $(K_{5}, E(K_5))$-minor. Then, $(G,\Sigma) \to SPC(2k-1)$.
\end{conjecture}

This motivates further studies of the simpler class of $(K_{4}, E(K_4))$-minor-free signed graphs (or odd-$K_4$-minor-free graphs). In particular, we would like to ask if in our main theorem (Theorem~\ref{thm:MinimalBound}), the condition of ``no $K_4$-minor'' for the underlying graph can be replaced by ``no $(K_{4}, E(K_4))$-minor'' for the signed graph itself.

\bigskip

{\bf Acknowledgments}\\ 
We thank the referees for carefully reading the original submission and helping to improve the presentation.

\end{document}